\documentclass[11pt,usletters]{amsart}
\usepackage{graphicx}
\usepackage[usenames]{color}
\usepackage{amsmath,amssymb,amsfonts}
\usepackage[colorlinks=true, pdfstartview=FitV, linkcolor=blue, citecolor=blue, urlcolor=blue,pagebackref=false]{hyperref}
\usepackage{mathrsfs}
\usepackage{enumerate}
\usepackage{cancel}

\numberwithin{equation}{section}

\newtheorem{theorem}{Theorem}[section]
\newtheorem{lemma}[theorem]{Lemma}
\newtheorem{proposition}[theorem]{Proposition}

\newtheorem{remark}{Remark}[section]

\newcommand{\R}{\mathbb R}
\newcommand{\C}{\mathbb C}

\newcommand{\eps}{\varepsilon}

\newcommand{\wt}{\widetilde}
\newcommand{\wh}{\widehat}
\newcommand{\half}{\frac{1}{2}}

\newcommand{\curl}{\nabla\times}
\newcommand{\dvg}{\nabla\cdot}
\newcommand{\grad}{\nabla}

\newcommand{\Lap}{\Delta}

\newcommand{\Ht}{{\tt H}}
\newcommand{\des}{\delta\sigma}
\newcommand{\den}{\delta n}
\newcommand{\deq}{\delta q}
\newcommand{\deE}{\delta E}
\newcommand{\deH}{\delta H}
\newcommand{\cout}[1]{}
\newcommand{\mA}{{\mathcal A}}

\newcommand{\mF}{{\mathcal F}}
\newcommand{\mH}{{\mathcal H}}
\newcommand{\mS}{{\mathcal S}}

\newcommand{\dd}{\mathrm{d}} % It is used to write the Lebesgue measure
\newcommand{\norm}[3]{\left\|#1\right\|^{#2}_{#3}} % It is used to denote the norm of functional spaces
\newcommand{\Duality}[2]{\left\langle #1 \Big| #2 \right\rangle} % It is used to denote the pairing of duality in displaymath
\newcommand{\inner}[2]{\left\langle #1, #2 \right\rangle} % It is used to denote the inner product of forms
\newcommand{\HOX}[1]{\marginpar{\footnotesize #1}}

\newcommand{\ndim}{\mathfrak{n}}

\author{Guillaume Bal%\thanks{Applied Mathematics, Columbia University ({\tt gb2030@columbia.edu}). This author is partly supported by NSF .} 
\and Ting Zhou%\thanks{Department of Mathematics, Massachusetts Institute of Technology ({\tt tzhou@math.mit.edu}). This author is supported by NSF grant DMS 6926215}
}
\title[]{Hybrid Inverse problems for a system of Maxwell's equations}

\begin{document}

\begin{abstract}
This paper concerns the quantitative step of the medical imaging modality Thermo-acoustic Tomography (TAT). We model the radiation propagation by a system of Maxwell's equations. We show that the index of refraction of light and the absorption coefficient (conductivity) can be uniquely and stably reconstructed from a sufficiently large number of TAT measurements. Our method is based on verifying that the linearization of the inverse problem forms a redundant elliptic system of equations. We also observe that the reconstructions are qualitatively quite different from the setting where radiation is %(hence incorrectly) 
modeled by a scalar Helmholtz equation as in \cite{BRUZ-IP-11}.
\end{abstract}

\maketitle

%\tableofcontents
\setcounter{tocdepth}{1}

\section{Introduction}

Thermo-acoustic tomography (TAT) is a medical imaging technique that belongs to the class of coupled-physics imaging modalities. As electromagnetic waves (micro-waves with wavelengths typically of order $1m$) propagate through a domain we wish to probe, some radiation is absorbed, heats up the underlying tissues, creates a small mechanical expansion and hence generates some ultrasound that propagates to the boundary of the domain where they are measured by transducers. The first step  of TAT consists of reconstructing the map of absorption radiation from the ultrasound measurements. It is typically modeled by an inverse source wave problem and has been analyzed in detail in many works that include \cite{FR-CRC-09,FPR-JMA-04,HSS-MMAS-05,HKN-IP-08,KK-EJAM-08,PS-IP-07,SU-IP-09,SU-IO-12,XW-RSI-06}. This paper concerns the second, quantitative, step, which aims to reconstruct the optical properties of the tissues from knowledge of the (non-quantitative) absorption maps obtained during the first step.

TAT is an example of a coupled-physics modality, which combines the high contrast of a physical phenomenon (here the electrical properties of tissues) with the high resolution of another phenomenon (here ultrasound). Other coupled-physics modalities have been explored experimentally and analyzed mathematically, sometimes under the name of hybrid inverse problems. For a very incomplete list of works on these problems in the mathematical literature, we refer the reader to  \cite{BR-IP-11,BRUZ-IP-11,BU-CPAM-12,BU-IP-10,BU-AML-12}  in the quantitative step of the imaging modalities Photo-acoustic tomography, Thermo-acoustic tomography, Transient Elastography, and Magnetic Resonance Elastography and to  \cite{ABCTF-SIAP-08,B-APDE-13,BBMT-13,BGM-IP-13,BGM-IPI-13,BM-LINUMOT-13,BS-PRL-10,CFGK-SJIS-09,GS-SIAP-09,KK-AET-11,MB-aniso-13,MB-IP-12,MB-IPI-12} for works in  Ultrasound Modulated tomography and in Current Density Imaging.

In TAT, radiation is modeled by an electromagnetic field that satisfies the time-harmonic system of Maxwell's equations  (assuming a constant magnetic permeability)
\begin{equation}\label{eqn:ME-1-0}
-\curl\curl E+(\omega^2 n +i\omega\sigma) E =0\quad\mbox{in }\;\Omega,\quad
E\times \nu_{|\partial\Omega} = f,
\end{equation}
where $\omega$ is the frequency, $n$ the index of refraction and $\sigma$ the conductivity, $\Omega$ an open bounded domain in $\R^3$ with boundary $\partial\Omega$, $\nu$ the outward unit normal on $\partial\Omega$, and $f$ a boundary condition (illumination). The amount of absorbed radiation by the underlying tissue is given by $H(x)\equiv H_f(x)=\sigma(x)|E|^2(x)$ for $x\in\Omega$. The quantitative step of TAT (QTAT) concerns the reconstruction of $(n(x),\sigma(x))$ from knowledge of $\{H_j(x)=H_{f_j}(x)\}_{1\leq j\leq J}$ obtained by probing the medium with the $J$ illuminations $f_j$.

Our main result is that for $J$ sufficiently large, $(n(x),\sigma(x))$ can be uniquely and stably reconstructed from $\{H_j\}_{1\leq j\leq J}$ with no loss of derivatives. We also note the following result. The reconstruction of $\sigma(x)$ with $n(x)$ constant and known was addressed in \cite{BRUZ-IP-11}. It was shown there that $\sigma$ could be uniquely and stably reconstructed from one (well-chosen) measurement $H(x)$ provided that $\sigma$ was sufficiently small (compared to $\omega$).  In that paper, it was also shown that in a scalar Helmholtz model for radiation propagation $\Delta u + (\omega^2n+i\omega\sigma(x))u=0$, $\sigma$ could uniquely be reconstructed from knowledge of $H(x)=\sigma(x)|u(x)|^2$ without any smallness constraint on $\sigma$. Surprisingly, the latter result does not extend to the setting of Maxwell's system \eqref{eqn:ME-1-0}. We rather obtain that the reconstruction of $\sigma$ from $H(x)=\sigma|E|^2$ is very similar to the inversion of the $0-$Laplacian that finds applications in ultrasound modulation and is analyzed in \cite{B-APDE-13}. There, it is shown that $\sigma$ is partially reconstructed from knowledge of $H$ with the loss of one derivative. 

Some of our main conclusions are therefore that: (i) the Helmholtz equation is a qualitatively different model for radiation propagation in the context of QTAT; and (ii) both $\sigma$ and $n$ can uniquely and stably be reconstructed from $\{H_j\}_{1\leq j\leq J}$ with no loss of derivatives for well-chosen illuminations $\{f_j\}_{1\leq j\leq J}$.

The rest of the paper is structured as follows. Section \ref{sec:results} introduces the mathematical problem and presents our main results.  The analysis of the linearized inverse problem is carried out in \ref{sec:ellipticity}. The proof of ellipticity used in section \ref{sec:ellipticity} hinges on the analysis of complex geometric optics solutions for the system of Maxwell's equations that is carried out in \ref{sec:CGO}. %{\color{red}Finally, local stability estimates for the full nonlinear problem are presented and derived in section \ref{sec:nonlinear}.}

%%%%%%%%%%%%%%%%
\section{Presentation of QTAT and main results}
\label{sec:results}
Let $\Omega$ be a bounded open subset in $\R^3$ with smooth boundary $\partial\Omega$. The propagation of electromagnetic radiation in TAT is modeled by the system of Maxwell's equations
\begin{equation}
\curl E=-\partial_tB,\qquad \curl \Ht=\partial_t D+J,\quad\mbox{in }\;\Omega
\end{equation}
where $(E, \Ht)$ are the electric and magnetic fields, $(B, D)$ are the magnetic and electric flux density and $J$ represents the electrical current density. We assume linear and isotropic constitutive relations 
\[D=\eps E,\quad B=\mu \Ht,\quad\mbox{and}\quad J=\sigma E,\]
where $(\eps(x), \mu, \sigma(x))$ are scalar functions that characterize relative electric permittivity, magnetic permeability and conductivity of the media. 
% Here ``relative" means it's the ratio $\eps/\eps_0$ where $\eps_0$ is the vacuum permittivity. So we can assume $\eps_0=1$, $\mu_0=1$ and $\sigma_0=0$.
Assuming that the magnetic permeability $\mu$ is constant and normalized to $1$, which is an excellent approximation in medical imaging applications, we recast the above system as 
\begin{equation}-\curl\curl E =\sigma\partial_tE+n\partial_t^2E\quad\mbox{in }\;\Omega,\end{equation}
where $n(x)$ can be understood as the square of the refractive index of the medium and satisfies $n={\mu\eps}$ in this context. With time-harmonic sources and the solution given by $E(t, x)=e^{i\omega t}E(x)$, we obtain the equation for $E(x)$
\begin{equation}\label{eqn:ME-1}
-\curl\curl E+(\omega^2 n +i\omega\sigma) E =0\quad\mbox{in }\;\Omega.
\end{equation}

Denote the set of admissible coefficients as
\[L_{\textrm{ad}}^\infty(\Omega):=\left\{(n,\sigma)\in L^\infty(\Omega;\R^2)~|~n(x)\geq C>0,~\sigma(x)\geq 0~\mbox{ for }~x\in\overline\Omega\right\}.\]
Imposing the boundary condition (illumination)
\begin{equation}\label{eqn:bdry-illum}\nu\times E|_{\partial\Omega}=f,\end{equation}
a standard well-posedness theory for Maxwell's equations \cite{colton-kress-98,dlen3} states that given $(n,\sigma)\in L_{\textrm{ad}}^\infty(\Omega)$ and $f\in TH^{1/2}(\partial\Omega)$, the equation \eqref{eqn:ME-1}-\eqref{eqn:bdry-illum} has a unique solution in $ H^1(\Omega;\C^3)$. Here $H^s(\Omega)$ denotes the standard Sobolev space and
\[TH^s(\partial\Omega):=\left\{f\in H^s(\partial\Omega;\C^3)~|~\nu\cdot f=0\right\}\]
where $\nu$ is the unit outer normal vector of $\partial\Omega$. %Throughout the rest of the paper, we will occasionally abuse the notation for vector function spaces, e.g., use $H^1(\Omega)$ to denote $(H^1(\Omega))^\ndim$, when there is no confusion.

Associated to each boundary illumination $f$, the internal map of absorbed radiation in $\Omega$, which is reconstructed from the first step of TAT by solving an inverse ultrasound problem, is given by 
\begin{equation}\label{eqn:H}
H(x)\equiv H_f(x)=\sigma(x)|E(x)|^2,\qquad x\in\Omega,
\end{equation}
where $E$ is the solution of \eqref{eqn:ME-1}-\eqref{eqn:bdry-illum}. The inverse problem of QTAT is then to reconstruct $(n(x),\sigma(x))$ in $\Omega$ from $\{H_j(x):=H_{f_j}(x)\}_{1\leq j\leq J}$ for $f_j$ in a properly chosen set of  boundary illuminations $\{f_j\}_{1\leq j\leq J}$ with $J\geq1$.

In \cite{BRUZ-IP-11}, a scalar model of the Helmholtz equation was studied in place of Maxwell's equations. Assuming $n$ known and constant, it was shown that $\sigma$ could be uniquely reconstructed from the nonlinear internal functional $H[\sigma](x)$ using a fixed-point algorithm.
In the same paper, it was shown that $\sigma$ could also be uniquely reconstructed from \eqref{eqn:ME-1}-\eqref{eqn:H} when $\sigma$ was sufficiently small. However, the proof for the Helmholtz case, based on specific properties of complex geometric optics (CGO) solutions for scalar equations did not extend to the Maxwell case for large values of $\sigma$. That result is in fact incorrect as we shall see. We follow here the method presented in \cite{B-Irvine-13}, see also \cite{BM-LINUMOT-13,KS-IP-12}, and first consider a linearized version of our nonlinear inverse problem. A more abstract functional analysis approach of such type is described in \cite{StU}. 

Let us define $q:=\omega^2n+i\omega\sigma$ and recast \eqref{eqn:ME-1} as 
\begin{equation}\label{eqn:ME-4}
(\Lap-\grad\dvg+q)E=0.\end{equation}
The above operator is not elliptic, but we can render it  elliptic by taking the divergence of the above equation and obtain the redundant system
\[
(\Lap-\grad\dvg+q)E=0,\qquad \dvg(qE)=0.
\]
To make all differential equations  second-order, we have
\begin{equation}\label{eqn:ME-3}
(\Lap-\grad\dvg+q)E=0,\qquad \grad\dvg(qE)=0.
\end{equation}
Assume that $q\neq0$ everywhere, which holds since the index of refraction $n$ does not vanish,
%\HOX{notice the assumption here!}
and rewrite the above as the elliptic system
\begin{equation}\label{eqn:ME-2}\left(\Lap+\frac{1}{q}[\grad\dvg,q]+q\right)E=0.\end{equation}
Here, $[A,B]:=AB-BA$ is the usual commutator. 

\begin{remark}
The above elliptic equation can be written in any dimension $\ndim\geq3$, in which case, the differential form notations shall be adopted. More specifically, let $E$ be a differential 1-form defined on $\Omega$, a bounded domain in $\R^\ndim$ ($\ndim\geq 3$). 
Equation \eqref{eqn:ME-2} can be written as 
\[\left(\Lap_H+\frac{1}{q}[d\delta,q]+q\right)E=0,\]
where $\Lap_H=\delta d+d\delta$ is the Hodge-Laplacian operator, $d$ is the exterior derivative and $\delta$ is the formal adjoint operator defined by $\delta =(-1)^{\ndim(l+1)+1}\ast d~\ast $ for $l$-forms. Here, $\ast$ is the Hodge-star operator. In particular, the commutator operator is given as
\[ [d\delta, q]=-d(dq~\vee)+dq\wedge\delta,\]
where $q$ is viewed as a 0-form and operators $\wedge$ and $\vee$ denote the exterior product and its adjoint respectively. %Note that in terms of differential forms, $\curl\curl$ is the operator $\delta d$ while $\grad\dvg$ is $d\delta$. 
%Moreover, in the least square approach with Cauchy boundary data, this problem is uniquely solvable by standard unique continuation principle.
\end{remark}

Consider the perturbation of $(n,\sigma)\in L^\infty_{\textrm{ad}}(\Omega)$ given by
$\tilde n=n+\epsilon \den,  \tilde \sigma=\sigma+\epsilon\des$,
where $\den,\des\in L^\infty_0(\Omega)$. Correspondingly, we denote $\tilde q=q+\epsilon\deq$ where $q=\omega^2 n+i\omega \sigma$ and $\deq=\omega^2\den+i\omega\des$. For $\epsilon>0$ small, the solution to \eqref{eqn:ME-4} with boundary condition $f_j$ can be written as $\wt E_j=E_j+\epsilon\deE_j+o(\epsilon)$ 
where $E_j$ is the solution with respect to $q$, and $\deE_j\in H_0^1(\Omega;\C^\ndim)$ satisfies
\[
\left(\Delta-\grad\dvg +q\right)\deE_j=-\deq E_j\quad\mbox{in }\;\Omega.
\]
The elliptic version of the above linearization is then given by
\begin{equation}\label{eqn:ME-l-e}
\left(\Delta+\frac{1}{q}[\grad\dvg, q]+q\right)\deE_j=-\deq E_j-\frac{1}{q}\grad\dvg(\deq E_j)\quad\mbox{in }\;\Omega.
\end{equation}
Taking complex conjugates yields
\begin{equation}\label{eqn:ME-l-e-conj}
\left(\Delta+\frac{1}{q^*}[\grad\dvg, q^*]+q^*\right)\deE_j^*=-\deq^* E_j^*-\frac{1}{q^*}\grad\dvg(\deq^* E_j^*)\quad\mbox{in }\;\Omega.
\end{equation}
For the internal functional $H_j=\sigma|\wt E_j|^2$, the Fr\'echet derivative $\deH_j$ obtained in the same manner is given by
\begin{equation}\label{eqn:H-l}
\deH_j=\sigma(\deE_j\cdot E_j^*+E_j\cdot\deE_j^*)+\des|E_j|^2,
\end{equation}
where $z^*$ denotes the complex conjugate of $z$. 

Consider the system of $3J$ differential equations \eqref{eqn:ME-l-e}, \eqref{eqn:ME-l-e-conj} and \eqref{eqn:H-l} for $1\leq j\leq J$ of $2J\ndim+2$ unknowns $\deE_j$, $\deE_j^*$ ($1\leq j\leq J$), $\den$ and $\des$. Our first main result concerns the ellipticity of a boundary value problem for this $(2J\ndim+J)\times (2J\ndim+2)$ matrix differential operator. 

\cout{  % I believe we do not need this
\HOX{go to Appendix?}
For the convenience of analysis, we adopt the standard notations of symbol classes $S^m(\R^d)\ni a(x,\xi)$ for $m$-th order pseudo-differential operators defined by
\[a(x,D)u=\frac{1}{(2\pi)^n}\int\int a(x,\xi)e^{i(x-y)\cdot\xi}u(y)~dy~d\xi\]
where $D=-i\grad$. (The symbol of laplacian is $-|\xi|^2$.)
The symbol of the composition of two operators $b(x,D)a(x,D)$ is hence given by 
\[b\sharp a(x,\xi):=b(x,\xi+D_x)a(x,\xi)=\frac{1}{(2\pi)^n}\int e^{iy\cdot\eta}b(x,\xi-\eta)a(x-y,\xi)~dy~d\eta\]
with principal part being $b(x,\xi)a(x,\xi)$. In particular, the symbols of differential operators are polynomials. A pseudo-differential operator is said to be {\em elliptic} if its principal symbol never vanishes for $x\in\R^d$ and $\xi\in\R^d\backslash\{0\}$. An elliptic symbol $a\in S^m(\R^d)$ admits a parametrix $b\in S^{-m}(\R^d)$, that is,
\[a\sharp b-I\;\mbox{ and }\; b\sharp a-I \;\in S^{-\infty}(\R^d).\]
One can easily generalize the definitions to the case of systems, where the composition $\sharp$ is compatible with corresponding matrix multiplications. 
}

\medskip

For a matrix-valued differential operator $\mA(x,D)=\left(a_{ij}(x,D)\right)_{1\leq i\leq i_{\textrm{max}}, 1\leq j\leq j_{\textrm{max}}}$ ($i_{\textrm{max}}\geq j_{\textrm{max}}$), the principal symbol and the notion of ellipticity are defined in the Douglis-Nirenberg sense in \cite{DN,ADN-ii} for square systems and \cite{So} for redundant systems. More precisely, we assign two sets of integers $\{s_i\}_{1\leq i\leq i_{\textrm{max}}}$ and $\{t_j\}_{1\leq j\leq j_{\textrm{max}}}$ such that $a_{ij}$ has order not greater than $s_i+t_j$ and $a_{ij}=0$ if $s_i+t_j<0$. Then the principal part of $\mA(x,D)$ is the matrix operator $\mA_0(x,D)$ with elements being the leading terms of those $a_{ij}$ whose order is exactly $s_i+t_j$. The operator $\mA(x,D)$ is said {\em elliptic} at $x\in\Omega$ if the rank of $\mA_0(x,\xi)$ equals to $j_{\textrm{max}}$ for all $\xi\in\R^\ndim\setminus\{0\}$ and is said to be elliptic on $\Omega$ if it is elliptic at every $x\in\Omega$.

%To formulate an elliptic boundary-value problem for an elliptic $m\times l$ matrix operator $A(x,D)$, we need to impose a complementing boundary condition as described in \cite{ADN-ii,So} (also known as the elliptic boundary conditions in \cite{Ho} or Lopatinskii condition \cite{Lo}).

Given a linear system $\mA(x,D)v(x)=\mS(x)$ in a bounded domain $\Omega$, with sets of integers $\{s_i\}$ and $\{t_j\}$ as above, we consider the boundary condition of the form
\begin{equation}\label{eq:bdry-cond-gen}\mathcal{B}(x,D)v(x)=g(x)\qquad x\in\partial\Omega
\end{equation}
where $\mathcal{B}(x,D)=\left(\mathcal B_{qj}(x,D)\right)_{1\leq q\leq Q, 1\leq j\leq l}$ is a matrix differential operator. The principal part $\mathcal B_{qj}^{(0)}$ of $\mathcal B_{qj}$ is defined as the sum of all terms of order $\varrho_q+t_j$, where $\varrho_q=\max_j(\beta_{qj}-t_j)$ and $\beta_{qj}$ is the order of $\mathcal B_{qj}$. Then the principal part $\mathcal{B}_0(x,D)$ of $\mathcal{B}(x,D)$ is the matrix with elements $\mathcal B_{qj}^{(0)}$.

According to \cite{So}, the condition \eqref{eq:bdry-cond-gen} is called the {\em complementing boundary condition} to the system, and we then say that $\mathcal{B}$ covers $\mA(x,D)$, if it satisfies the Lopatinskii criterion \cite{Lo}: for any $y\in\partial\Omega$ and a nonzero tangential vector $\zeta$ at $y$ to $\partial\Omega$, the one dimensional boundary-value problem
\begin{equation}\label{eqn:BVP-1d}
\left\{\begin{split}
& \mA_0\big(y,\zeta+i\nu(y)\frac{d}{dz}\big)\wt u(z)=0,\qquad z>0,\\
&\mathcal B_0\big(y,\zeta+i\nu(y)\frac{d}{dz}\big)\wt u(z)\Big|_{z=0}=0,\\
&\wt u(z)\rightarrow 0,\qquad\mbox{as }z\rightarrow\infty
\end{split}\right.\end{equation}
has a unique solution $u=0$. Here $\nu(y)$ is the outer normal unit vector to $\partial\Omega$. %and $z$-axis is along the direction $-\nu(y)$.\\

We say that a boundary value problem is {\em elliptic} if the corresponding matrix operator is elliptic and the boundary condition is complementing.

In \cite{ADN-ii, So}, Schauder type estimates are generalized to the setting of elliptic systems of boundary value problems. For our hybrid inverse problems, these regularity estimates indicates how errors in the internal maps propagate to errors in the reconstruction of the parameters, provided that we can show (independently) the uniqueness of the reconstruction.
%\HOX{by the open mapping theorem, see summer program report at Fields Institute.}
To this end, we apply the differential operator $\Lap$ to \eqref{eqn:H-l} and using \eqref{eqn:ME-l-e} and \eqref{eqn:ME-l-e-conj}, obtain that
\begin{equation}\label{eqn:Lap-H-l}
|E_j|^2\Lap\des-\frac{\sigma_0}{q_0}E_j^*\cdot(E_j\cdot\grad\grad\deq)-\frac{\sigma_0}{q_0^*}E_j\cdot(E_j^*\cdot\grad\grad\deq^*)+\textrm{l.o.t.}=\Lap\deH_j,\end{equation}
where l.o.t. represents lower order differential terms (of order $1$ or $0$). We point out that \eqref{eqn:Lap-H-l} can be obtained by directly linearizing $\Delta H_j$. In other words, we have obtained a linearization of the nonlinear problem
\begin{equation}
\label{eq:nonlinemod}
   \begin{array}{rcl}
\left(\Lap+\frac{1}{q}[\grad\dvg,q]+q\right)E_j&=&0 \\
  \Delta ( \sigma |E_j|^2 ) &=& \Delta H_j.
\end{array}
\end{equation}
With $v=(\{E_j, E_j^*\}_{j=1}^J, \sigma,n)$, still using $(E_j,E_j^*)$ instead of the (real-valued) real and imaginary parts, we recast the above nonlinear problem as
\begin{equation}
\label{eq:nonlintF}
\tilde \mF (v) = \mH
\end{equation}

The leading term of \eqref{eqn:Lap-H-l} can be expanded as
\begin{equation}\label{eqn:Lap-H-l-p}
\left(|E_j|^2\Lap-\frac{2\omega^2\sigma_0^2}{|q_0|^2}E_j\otimes E_j^*:\grad^{\otimes 2}\right)\des-\frac{2\omega^4\sigma_0n_0}{|q_0|^2}E_j\otimes E_j^*:\grad^{\otimes 2}\den
\end{equation}
where $a\otimes b:=ab^T$ for vectors and $A:B$ denotes $\displaystyle\sum_{i,j}A_{ij}B_{ij}$ for matrices.  

The above expression is the main ingredient that we use to prove that the complete system for $w:=(\{\deE_j, \deE_j^*\}_{j=1}^J, \des,\den)$ is elliptic for an open set of boundary conditions $\{f_j\}$ used to construct the solutions $\{E_j\}$. Because the system is second-order for $w$, it requires boundary conditions, chosen here as Dirichlet conditions $w^\delta := (\{{\deE_j}_{|\partial\Omega}, {\deE_j^*}_{|\partial\Omega}\}_{j=1}^J,\des_{|\partial\Omega},\den_{|\partial\Omega})$ known on $\partial \Omega$, that are more constraining than the conditions prescribed on $\{E_j\}$. These boundary conditions need to be measured in practice, and we expect the errors made on such measurements to be small if the $v=(\{E_j, E_j^*\}_{j=1}^J, \sigma,n)$  about which we linearize are close to the true value of these parameters. We collect the equations \eqref{eqn:ME-l-e}, \eqref{eqn:ME-l-e-conj} and \eqref{eqn:Lap-H-l} into the linear system 
\begin{equation}
\label{eq:linmod}
 \mA w = \mS\quad \mbox{ in } \quad \Omega,%\qquad w|_{\partial \Omega} = w^\partial \quad \mbox{ on } \quad\partial\Omega.
\end{equation}

With this, we obtain the following elliptic regularity result.
\begin{theorem}\label{thm:main-1}
Let $\Omega$ be a bounded open subset of $\R^3$ with smooth boundary.
Given $J=4$, there exists a boundary illumination set $\{f_j\}_{j=1}^4\subset \textrm{TH}^{1/2}(\partial\Omega)$ such that the redundant linear system \eqref{eq:linmod} for $w:=(\{\deE_j, \deE_j^*\}_{j=1}^J, \des,\den)$ augmented with the Dirichlet boundary condition % \HOX{Aren't we supposed to put zero Dirichlet boundary condition here?}
\begin{equation}\label{eqn:BC-D}
w_{|\partial\Omega}= w^\delta,
%\deE_j|_{\partial\Omega}=0,\quad (j=1,\ldots,J),\quad \des=0,\quad \den=0.
\end{equation}
is an elliptic boundary value problem.  
Moreover, we have the following estimate
\begin{equation}\label{eqn:stability}
\begin{split}
\|w\|_{H^s(\Omega;\C^{6J+2})} \leq C_1\Big( \|\delta H\|_{H^s(\Omega;\R^{J})} + \|w^\delta\|_{H^{s-\frac12}(\partial\Omega;\C^{6J+2})}\Big) + C_2 \|w\|_{L^2(\Omega;\C^{6J+2})},
%&\sum_{j=1}^J\|\deE_j\|_{\left(H^{l'+2}(\Omega)\right)^3}+\|\des\|_{H^{l'+2}(\Omega)}+\|\den\|_{H^{l'+2}(\Omega)}\\
%&\leq C_1\sum_{j=1}^J\|\Lap\deH_j\|_{H^{l'}(\Omega)}+C_2\left(\sum_{j=1}^J\|\deE_j\|_{L^2(\Omega)}+\|\des\|_{L^2(\Omega)}+\|\den\|_{L^2(\Omega)}\right)
\end{split}
\end{equation}
for all $s>2+\frac32=\frac72$ provided that $(\sigma,n,\{E_j\})$ are sufficiently smooth. % I believe $C^2$ is more than enough, or use $H^s$. 
%\HOX{fill in later} 
\end{theorem}
The above result states that four well-chosen boundary illuminations would provide ellipticity of the linearized system of equations. Its proof is based on the elliptic regularity Theorem 1.1 of \cite{So} derived for redundant elliptic systems. We need to ensure that the assumptions of ellipticity are satisfied, and this will be the main objective of the next two sections.

%  To obtain the elliptic operator, the four illuminations are chosen to be traces of near-plane waves, that is, CGO solutions of the background system.\\

\begin{remark}
Before presenting injectivity and stability results for the linear and nonlinear problems, let us pause on the relation between the vectorial problem based on the Maxwell's system of equations and the scalar problem based on the Helmholtz equation. Let us assume that $n$ is known so that $\delta n=0$ in the above equations. Some analysis that easily follows from calculations in \cite{BRUZ-IP-11} shows that the equivalent statement to \eqref{eqn:Lap-H-l-p} for the scalar case is
\begin{displaymath}
  |u|^2 \delta\sigma = \delta H + \mbox{ lower order terms}.
\end{displaymath}
Since $|u|^2>0$, the above equation for $\delta\sigma$ is clearly elliptic (without any additional boundary condition on $\partial \Omega$). This is consistent with (though clearly not equivalent to) the results obtained in \cite{BRUZ-IP-11}. In contrast \eqref{eqn:Lap-H-l-p}  for the Maxwell case with $\den=0$ is given by
\begin{displaymath}
  \left(|E_j|^2\Lap-\frac{2\omega^2\sigma_0^2}{|q_0|^2}E_j\otimes E_j^*:\grad^{\otimes 2}\right)\des = \Delta H_j +  \mbox{ lower order terms}.
\end{displaymath}
The symbol of the above principal  term is given by
\begin{displaymath}
  |\xi|^2 - \tau \hat E_j\otimes \hat E_j : \xi\otimes \xi ,\qquad \tau = \dfrac{2\omega^2\sigma_0^2}{\omega^2\sigma_0^2+\omega^4 n_0^2},\quad \hat E_j=\dfrac{E_j}{|E_j|}.
\end{displaymath}
When $\tau<1$, then the above operator is clearly elliptic. However, for $\tau>1$, the operator becomes hyperbolic with respect to the direction $\hat E_j$ so that the above operator becomes a wave-type operator and elliptic regularity no longer holds. We thus observe a qualitative difference between the scalar Helmholtz and vectorial Maxwell problems from the reconstruction of $\sigma$ from knowledge of the absorption map $H(x)$.
\end{remark}

\medskip

%IT REMAINS TO ADD SOMETHING ON INJECTIVITY AND THE NONLINEAR PROBLEM.

Let us come back to the nonlinear hybrid inverse problem and the above theorem. The presence of the constant $C_2\not=0$ indicates that the linearized problem may not be injective. Following the methodology presented in \cite{B-Irvine-13}, we introduce the linearized normal operator defined as
\begin{equation}
\label{eq:normalop}
 \mA^t \mA w = \mA^t \mS\quad\mbox{ in } \quad\Omega,\qquad w_{|\partial \Omega} = w^\delta\ \mbox{ and } \ \partial_\nu w_{|\partial\Omega} = j^\delta.
\end{equation}
Since $\mA^t\mA$ is a fourth-order operator, we need two boundary conditions, which we consider of Dirichlet type. It is shown in \cite{B-Irvine-13} that the above linear operator is injective (i) when the coefficients $v_0=(\{E_j, E_j^*\}_{j=1}^J, \sigma,n)$ are in a sufficiently small vicinity of an analytic coefficient (with the vicinity depending on that analytic coefficients); and (ii) when the domain $\Omega$ is sufficiently small. 

The stability estimate presented in the above theorem then extends to a nonlinear hybrid inverse problem as follows. Let $v_0$ be given and $\mA$ be the linear operator defined above with coefficients described by $v_0$. Let us consider the nonlinear inverse problem
\begin{equation}
\label{eq:nonlinfinalmodif}
  \mF(v) := \mA^t \tilde\mF(v)   = \mA^t \mH\quad\mbox{ in }\quad\Omega,\qquad v_{|\partial \Omega} = v^\delta\ \mbox{ and } \ \partial_\nu v_{|\partial\Omega} = g^\delta.
\end{equation}
Then defining $v=v_0+w$ and linearizing the above inverse problem about $v_0$, we observe that the linear equation for $w$ is precisely of the form \eqref{eq:normalop}. This and the theory presented in \cite{B-Irvine-13} allow us to obtain the following result.
\begin{theorem}
\label{thm:stabnonlin}
  Let $v_0$ be defined as above and let us assume that the linear operator defined in \eqref{eq:normalop} is injective. Let $v$ and $\tilde v$ be solutions of \eqref{eq:nonlinfinalmodif} with respective source terms $\mH$ and $\tilde\mH$ and respective boundary conditions $v^\delta$ and $\tilde v^\delta$ as well as $g^\delta$ and $\tilde g^\delta$. Then $(\mH,v^\delta,g^\delta)=(\tilde\mH,\tilde v^\delta,\tilde g^\delta)$ implies that $v=\tilde v$, in other words the nonlinear hybrid inverse problem is injective. Moreover, we have the stability estimate
  \begin{equation}
\label{eq:stabnon}
  \|v-\tilde v\|_{H^s(\Omega;\C^{6J+2})} \leq C \Big( \|H-\tilde H\|_{H^s(\Omega;\R^J)} + \|v^\delta-\tilde v^\delta\|_{H^{s-\frac12}(\partial\Omega;\C^{6J+2})} +  \|g^\delta-\tilde g^\delta\|_{H^{s-\frac32}(\partial\Omega;\C^{6J+2})} \Big).
\end{equation}
This estimates hold for $C=C_s$ when $s>\frac72$.
\end{theorem}
The proof of this theorem is a direct consequence of Theorem \ref{thm:main-1} and the theory presented in \cite{B-Irvine-13}. It thus remains to prove the latter theorem, which is the object of the following two sections.

%%%%%%%%%%%%%%%%%
\section{Ellipticity}\label{sec:ellipticity}
Consider the system of equations \eqref{eqn:ME-l-e}, \eqref{eqn:ME-l-e-conj} and \eqref{eqn:Lap-H-l} for $1\leq j\leq J$ written in the form $\mA(x,D) v=b$ where 
\begin{equation}\label{eqn:X-b}\begin{split}
v&=\left(\deE_1, \deE^*_1, \ldots, \deE_J, \deE^*_J, \des, \den\right)^T,\\
b&=\left(\vec0,\ldots,\vec 0,\Lap\deH_1,\ldots,\Lap\deH_J\right)^T
\end{split}\end{equation}
and $\mA(x,D)$ is a second order matrix differential operator with the principal part given by
\begin{equation}\label{eqn:A0}\mA_0(x,D)=\left(\begin{array}{c|c}\Lap I_{2J\ndim} & \mA_{12}(x,D)  \\ \hline 0 & \begin{array}{cc}a_1(x,D) & b_1(x,D) \\ \vdots & \vdots \\a_J(x,D) & b_J(x,D)\end{array} \end{array}\right)\end{equation}  
where $I_\ndim$ is the $\ndim\times \ndim$ identity matrix. Here $a_j(x,D)$ and $b_j(x,D)$ are second order operators from \eqref{eqn:Lap-H-l-p} whose symbols are
\[a_j(x,\xi)= -|E_j|^2|\xi|^2+2\kappa|E_j\cdot\xi|^2,\qquad b_j(x,\xi)=2\tau|E_j\cdot\xi|^2\]
with $\kappa:=\frac{\omega^2\sigma_0^2}{|q_0|^2}$ and $\tau:=\frac{\omega^4\sigma_0n_0}{|q_0|^2}$.
For $\mA_0(x,\xi)$ to have rank $2J\ndim+2$ ($J\geq2$ so that the system is not underdetermined) for $x\in\Omega$ and $\xi\in\R^3\setminus\{0\}$, one has to show that the rank of 
\[\mA_{22}(x,\xi):=\left(\begin{array}{cc}a_1(x,\xi) & b_1(x,\xi) \\ \vdots & \vdots \\a_J(x,\xi) & b_J(x,\xi)\end{array}\right)\] is 2. This is equivalent to show that
\[(a_jb_l-a_lb_j)(x,\xi)=0\quad 1\leq j<l\leq J\quad\Rightarrow\quad\xi=0\]
for every $x\in\Omega$, that is,
\begin{equation}\label{eqn:ellip-cond-1}
|E_j|^2|E_l|^2\left(|\wh E_j\cdot\xi|^2-|\wh E_l\cdot\xi|^2\right)=0\quad 1\leq j<l\leq J\quad \Rightarrow\quad \xi=0
\end{equation}
where $\wh E_j:=E_j/|E_j|$.
Then the question is whether we can find enough background solutions, namely, solutions to
\begin{equation}\label{eqn:bkgd}
-\curl\curl E_j + q_0E_j=0,
\end{equation}
such that \eqref{eqn:ellip-cond-1} is satisfied. \\

If $q_0$ is a nonzero constant, \eqref{eqn:bkgd} degenerates to 
\[(\Lap+q_0)E_j=0,\quad \grad\dvg E_j=0\]
which admits plane wave solutions of the form $E_j=\eta_je^{i\zeta_j\cdot x}$ provided that
\begin{equation}-\zeta_j\cdot\zeta_j+q_0=\zeta_j\cdot \eta_j=0, \quad \zeta_j,\; \eta_j\in\C^\ndim\setminus\{0\}.\end{equation}
This allows us to choose $\eta_j$ to be any real vector, and $\textrm{Re}(\zeta_j)$ and $\textrm{Im}(\zeta_j)$ satisfying
\[\textrm{Re}(\zeta_j)\perp \eta_j, \quad \textrm{Im}(\zeta_j)\perp \eta_j,\]
\[\textrm{Re}(\zeta_j)^2-\textrm{Im}(\zeta_j)^2=\textrm{Re}(q_0), \quad 2\textrm{Re}(\zeta_j)\textrm{Im}(\zeta_j)=\textrm{Im}(q_0).\]
Such condition can be fulfilled in three or higher dimensional space.

Note that $|E_j|\neq 0$ everywhere and $|\wh E_j\cdot\xi|=|\wh \eta_j\cdot\xi|$ for $1\leq j\leq J$. In $\R^\ndim$, it requires $J\geq \ndim+1$ and $\textrm{span}\{\wh \eta_1,\ldots, \wh \eta_J\}=\R^\ndim$ to have 
\[|\wh \eta_j\cdot\xi|=|\wh \eta_l\cdot\xi|\quad 1\leq j< l\leq J \quad\Rightarrow\quad \xi=0,\]
which is condition \eqref{eqn:ellip-cond-1}. In particular, in $\R^3$, we have $J\geq 4$. 

Now we consider the non-constant $q_0$ case. The background wave we are using to fulfill the elliptic condition is the CGO solution, originally constructed for the full Maxwell's equations in \cite{OPS,OS}. We present in the following lemma the CGO electric field corresponding to our system, with extra point-wise estimates at our disposal. The proof of the estimates is detailed in Section \ref{sec:CGO}. Here we denote by $B_r$ a ball in $\R^\ndim$ of radius $r>0$. 
\begin{lemma}\label{lem:CGO} Let $\Omega$ be a bounded domain in $\R^\ndim$ with smooth boundary. 
Let $q_0=\omega^2n_0+i\omega\sigma_0\in {H^{l'+2}(\R^\ndim)}$ ($l'>\ndim/2$)  
and $q_0(x)\neq 0$ everywhere. Suppose that  $n_0(x)-n_c\geq0$ and $\sigma_0(x)\geq0$ are compactly supported on some ball $B\supset\supset\overline{\Omega}$ for some constant $n_c>0$. 
Denote $k=\omega \sqrt{n_c}$. For $\rho, \rho^\bot\in\mathbb S^{\ndim-1}$ with $\rho\perp\rho^\bot$ and $s>0$, define
\begin{equation}\label{eqn:zeta}
\zeta:=-is\rho+\sqrt{s^2+k^2}\rho^\bot
\end{equation}
and 
\begin{equation}\label{eqn:eta-zeta}
\eta_\zeta:=\frac{1}{|\zeta|}\left(-(\zeta\cdot \vec a)\zeta-k\zeta\times \vec b+k^2 \vec a\right)
\end{equation}
for any $\vec a, \vec b\in\C^\ndim$. Then there exists an $s_0>0$ such that for $s>s_0$, the equation \eqref{eqn:bkgd} admits a unique solution in $H^1_{\textrm{loc}}(\R^\ndim;\C^{\ndim})$ of the form
\begin{equation}\label{eqn:E-CGO}
E_\zeta(x)=\gamma_0^{-1/2}e^{ix\cdot\zeta}(\eta_\zeta+R_\zeta(x))
\end{equation}
where $\gamma_0=q_0/k^2$ and $\gamma_0^{1/2}$ denotes the principal branch. Moreover, $R_\zeta\in L^\infty(\Omega;\C^\ndim)$ satisfies
\begin{equation}\label{eqn:CGO-R}\|R_\zeta\|_{L^\infty(\Omega;\C^\ndim)}\leq C \end{equation}
where $C>0$ is independent of $s$. 
\end{lemma}

Let us denote 
\[\wh \zeta_\infty=\lim_{s\rightarrow\infty}\frac{\zeta}{|\zeta|}=\frac{1}{\sqrt 2}(-i\rho+\rho^\perp)\]
and take $\vec a$ such that $\wh \zeta_\infty\cdot \vec a =1$, for example, $\vec a=\wh\zeta_\infty^*$. We can choose $\vec b\in\mathbb S^2$ relatively freely.  It is not hard to see that as $s\rightarrow\infty$,
\[|\eta_\zeta +\zeta|=o(s).\]
Note that $|\zeta|=\sqrt{2s^2+k^2}\sim\sqrt 2s$ as $s\rightarrow\infty$.
Then \eqref{eqn:CGO-R} implies that for every $x\in\Omega$,
\[|E_\zeta(x)|\sim |\gamma_0(x)|^{-1/2}e^{s x\cdot\rho}\sqrt 2 s\quad\mbox{ as }\;s\rightarrow\infty.\]
This shows that $|E_\zeta(x)|\neq 0$ everywhere in $\Omega$. Moreover, 
\[|\wh E_\zeta(x)\cdot\xi|=\frac{\left|\left(\eta_\zeta+R_\zeta(x)\right)\cdot\xi\right|}{|\eta_\zeta+R_\zeta(x)|}\sim |\wh\zeta_\infty\cdot\xi|=\frac{1}{\sqrt 2}\sqrt{|\xi\cdot\rho|^2+|\xi\cdot\rho^\perp|^2} \]
independent of $x\in\Omega$ for $s$ large enough.

Now we choose $\ndim$ pairs of $(\rho_j, \rho_j^\perp)\in \mathbb S^{\ndim-1}\times\mathbb S^{\ndim-1}$ ($j=1,\ldots,\ndim$) such that $\rho_j=\wh e_\ndim$ and $\rho_j^\perp$ are $\ndim$ distinct directions in $\textrm{span}\{\wh e_1,\ldots,\wh e_{\ndim-1}\}$. Denote by $E_j$ the CGO solution corresponding to $(\rho_j,\rho_j^\perp)$ and some $s>0$. Then when $\xi\neq 0$ is not orthogonal to the $(\wh e_1,\ldots, \wh e_{\ndim-1})$-plane, at least two of $\{|\xi\cdot\rho_j^\perp|\}_{j=1,\ldots,\ndim}$ are not equal. This implies that for $\tau$ large enough, at least two of $\{|\wh E_j\cdot\xi|\}_{j=1,\ldots,\ndim}$ are not equal. 

At last, we choose the $(\ndim+1)$-th pair of directions $(\rho_{\ndim+1}, \rho_{\ndim+1}^\perp)=(\wh e_1, \frac{1}{\sqrt 2}\wh e_2+\frac{1}{\sqrt 2}\wh e_\ndim)$. Then in the case when $\xi\perp \textrm{span}\{\wh e_1,\ldots,\wh e_{\ndim-1}\}$, we have 
\[\sqrt{|\xi\cdot\rho_j|^2+|\xi\cdot\rho_j^\perp|^2}=|\xi\cdot \wh e_\ndim|=|\xi|\neq \frac{1}{\sqrt 2}|\xi|=\sqrt{|\xi\cdot\rho_{\ndim+1}|^2+|\xi\cdot\rho_{\ndim+1}^\perp|^2}\]
for $j=1,\ldots,\ndim$.
Therefore, for $\tau$ large enough, $|\wh E_{\ndim+1}\cdot\xi|\neq |\wh E_{j}\cdot\xi|$ ($j=1,\ldots,\ndim$), where $E_{\ndim+1}$ is the CGO solution for $(\rho_{\ndim+1}, \rho_{\ndim+1}^\perp)$.
Hence the condition \eqref{eqn:ellip-cond-1} is verified and we have 
\begin{lemma}\label{lem:ellip-op}
Given $(\ndim+1)$ CGO solutions $E_j$ ($0\leq j\leq \ndim+1$) as above, for $s>0$ large enough, we have that the operator $\mA(x,D)$ is elliptic for $x\in\Omega$. 
\end{lemma}

\begin{remark}\label{rmk:pert}
It is not hard to see that with boundary illuminations from a small neighborhood (in the $H^{l'+3/2}(\partial\Omega;\C^\ndim)$ topology for example) of $\{E_j|_{\partial\Omega}\}$, the operator $\mA$ remains elliptic. 
\end{remark}

Now it remains to verify that the Dirichlet boundary condition \eqref{eqn:BC-D} is the complementing boundary condition to the system $\mA v=b$. %that is, given any $y\in\partial\Omega$ and $\zeta$ a nonzero tangential vector at $y$ the boundary-value problem \eqref{eqn:BVP-1d}  has a unique zero solution, where $A_0$ is given by \eqref{eqn:A0} and $\mathcal{B}(x,D)$ is the $J(2d+1)\times J(2d+1)$ identity matrix.
 
At this point, we need to write down explicitly the symbol of the second order operator $\mA_{12}$ in \eqref{eqn:A0}
\[\mA_{12}(x,\xi)=\left(\begin{array}{cc}-\frac{i\omega}{q_0}(E_1\cdot\xi)\xi & -\frac{\omega^2}{q_0}(E_1\cdot\xi)\xi \\ \frac{i\omega}{q_0^*}(E_1^*\cdot\xi)\xi& -\frac{\omega^2}{q_0^*}(E_1^*\cdot\xi)\xi\\ \vdots & \vdots \\ -\frac{i\omega}{q_0}(E_J\cdot\xi)\xi & -\frac{\omega^2}{q_0}(E_J\cdot\xi)\xi \\ \frac{i\omega}{q_0^*}(E_J^*\cdot\xi)\xi& -\frac{\omega^2}{q_0^*}(E_J^*\cdot\xi)\xi\end{array}\right)_{2J\ndim\times 2}.\]
And we point out that  with the choice of  $E_j$ as above, $\mA_{22}(y,\xi)$ is also full rank for all the boundary points $y\in\partial\Omega$.

Fixing $y\in\partial\Omega$ and $\zeta\in \{\nu(y)\}^\perp\setminus\{0\}$, consider the one dimensional boundary-value problem 
\begin{equation}\label{eqn:BVP-ODE}\left\{\begin{split}&\mA_0(y,\zeta+i\nu\partial_z) u(z)=:P_2 u''+P_1 u'+P_0 u=0\qquad z>0,\\ 
& u(0)=0,\qquad\lim_{z\rightarrow\infty} u(z)=0\end{split}\right.\end{equation}
with constant coefficient matrices (depending on $y$ and $\zeta$) given by
\[\begin{split}P_2(y,\zeta)&=\mA_0(y,i\nu)=\left(\begin{array}{c|c}I_{2J\ndim} & -\mA_{12}(y,\nu) \\ \hline0 & -\mA_{22}(y,\nu)\end{array}\right)\\
P_1(y,\zeta)&=\left(\begin{array}{c|c}0 & i\left[\mA_{12}(y;\zeta,\nu)+\mA_{12}(y;\nu,\zeta)\right] \\ \hline0 & i\left[\mA_{22}(y;\zeta,\nu)+\mA_{22}(y;\nu,\zeta)\right]\end{array}\right) \\
P_0(y,\zeta)&=\mA_0(y,\zeta)=\left(\begin{array}{c|c}-|\zeta|^2I_{2J\ndim} & \mA_{12}(y,\zeta) \\ \hline0 & \mA_{22}(y,\zeta)\end{array}\right)\end{split}\]
where 
\[\mA_{12}(y;\zeta,\nu):=\left(\begin{array}{cc}-\frac{i\omega}{q_0}(E_1\cdot\zeta)\nu & -\frac{\omega^2}{q_0}(E_1\cdot\zeta)\nu \\ \frac{i\omega}{q_0^*}(E_1^*\cdot\zeta)\nu& -\frac{\omega^2}{q_0^*}(E_1^*\cdot\zeta)\nu \\ \vdots & \vdots \\ -\frac{i\omega}{q_0}(E_J\cdot\zeta)\nu & -\frac{\omega^2}{q_0}(E_J\cdot\zeta)\nu \\ \frac{i\omega}{q_0^*}(E_J^*\cdot\zeta)\nu& -\frac{\omega^2}{q_0^*}(E_J^*\cdot\zeta)\nu\end{array}\right)_{2J\ndim\times 2},\]
\[\mA_{22}(y;\zeta,\nu):=\left(\begin{array}{cc}2\kappa (E_1\cdot\zeta)(E_1^*\cdot\nu) & 2\tau(E_1\cdot \zeta)(E_1^*\cdot\nu) \\ \vdots & \vdots \\2\kappa (E_J\cdot\zeta)(E_J^*\cdot\nu) & \tau(E_J\cdot \zeta)(E_J^*\cdot\nu)\end{array}\right)_{J\times 2}.\]

\begin{lemma}\label{lem:ellip-bc}
The system \eqref{eqn:BVP-ODE} has no non-trivial solutions. Therefore, the Dirichlet boundary condition \eqref{eqn:BC-D} is a complementing boundary condition to the operator $\mA(x,D)$ obtained from \eqref{eqn:ME-l-e}, \eqref{eqn:ME-l-e-conj} and \eqref{eqn:Lap-H-l} for $1\leq j\leq J$. 
\end{lemma}

\begin{proof}

Notice that the last $(J\times 2)$ equations of the system \eqref{eqn:BVP-ODE} are decoupled as equations of $u_{\ndim-1}$ and $u_{\ndim}$, the last two components of $u$. With $E_j$ ($j=1,\ldots,J$) chosen to be the CGO solutions as above, for every $y\in\partial\Omega$, condition \eqref{eqn:ellip-cond-1} implies that $\mA_{22}(y,\nu)$ has full rank 2. Without loss of generality, we can assume 
\[\wt \mA_{22}(y,\nu):=\left(\begin{array}{cc}a_1(y,\nu) & b_1(y,\nu) \\a_2(y,\nu) & b_2(y,\nu)\end{array}\right)\]
is non-singular. Similarly, denote by $\wt \mA_{22}(y;\zeta,\nu)$ and $\wt \mA_{22}(y;\zeta)$ the matrices of the first two rows of $\mA_{22}(y;\zeta,\nu)$ and $\mA_{22}(y;\zeta)$. 
Then it is sufficient to show that $\wt u:=(u_{\ndim-1}, u_\ndim)^T=0$ is the unique solution to
\begin{equation}\label{eqn:ODE-wtu}
\left\{\begin{split}&-\wt \mA_{22}(y,\nu)\wt u''+i[\wt \mA_{22}(y;\zeta,\nu)+\wt \mA_{22}(y;\nu,\zeta)]\wt u'+\wt \mA_{22}(y,\zeta)\wt u=0\quad z>0,\\ &\wt u(0)=0,\qquad \lim_{z\rightarrow\infty}\wt u(z)=0,\end{split}\right.
\end{equation} 
as then \eqref{eqn:BVP-ODE} is reduced to solving $u_k''-|\zeta|^2u_k=0$, $u_k(0)=u_k(\infty)=0$ for the first $2J\ndim$ components $u_1,\ldots,u_{2J\ndim}$ of $u$. 
We follow the standard procedure of solving the ODE \eqref{eqn:ODE-wtu} by looking for the eigenvalues, that is, the $\lambda$ s.t., 
\[\det\left\{-\lambda^2\wt \mA_{22}(\nu)+i\lambda[\wt \mA_{22}(\zeta,\nu)+\wt \mA_{22}(\nu,\zeta)]+\wt \mA_{22}(\zeta)\right\}=0.\]
Here we omit the $y$-dependence for simple notations. We obtain
\[\lambda_{1,2}=\pm|\zeta|,\quad \lambda_{3,4}=\frac{\mathfrak bi\pm\sqrt{-\mathfrak b^2+4\mathfrak a\mathfrak c}}{2\mathfrak a}\]
where $\mathfrak a, \mathfrak b, \mathfrak c$ are real and given by
\[\begin{split} \mathfrak a&=|\wh E_1\cdot\nu|^2-|\wh E_2\cdot\nu|^2\\
\mathfrak b&=2\textrm{Re}\left\{(\wh E_1\cdot\zeta)(\wh E_1^*\cdot\nu)-(\wh E_2\cdot\zeta)(\wh E_2^*\cdot\nu)\right\}\\
\mathfrak c&=|\wh E_1\cdot\zeta|^2-|\wh E_2\cdot\zeta|^2.
\end{split}\]
Again, w.l.o.g, we could assume $\mathfrak a>0$. 
Then the general solution is given by $\wt u=\displaystyle\sum_{l=1}^4c_je^{\lambda_j t}\vec v_j$ where $c_j\in\C$ and $\vec v_j$ is the complex eigenvector associated to $\lambda_j$, that is, such that
\[\left\{-\lambda_j^2\wt A_{22}(\nu)+i\lambda_j[\wt A_{22}(\zeta,\nu)+\wt A_{22}(\nu,\zeta)]+\wt A_{22}(\zeta)\right\}\vec v_j=0\]
If $-\mathfrak b^2+4\mathfrak a\mathfrak c<0$, as can be shown in $\R^\ndim$ with $\ndim\leq 3$, $\lambda_{3,4}$ are pure imaginary and the boundary condition implies that $\wt u=0$. On the other hand, if $-\mathfrak b^2+4\mathfrak a\mathfrak c>0$, $\lambda_2=-|\zeta|<0$ and $\textrm{Re}(\lambda_4)=-\frac{\sqrt{-\mathfrak b^2+4\mathfrak a\mathfrak c}}{2\mathfrak a}<0$ and the boundary condition gives $c_1=c_3=0$ and $c_2\vec v_2+c_4\vec v_4=0$.
It is then easy to show that $\vec v_2$ and $\vec v_4$ are linearly independent, hence $c_2=c_4=0$, which concludes the proof.

\end{proof}

\begin{proof}[Proof of {\bf Theorem} \ref{thm:main-1}]\quad
This is a direct corollary of Lemma \ref{lem:ellip-op}, Lemma \ref{lem:ellip-bc} and Theorem 1.1 of \cite{So}. 

\end{proof}

\section{Conditions for Ellipticity: estimates for CGO solutions}\label{sec:CGO}

In this section, we show the proof of Lemma \ref{lem:CGO}, that is, to construct the CGO solution \eqref{eqn:E-CGO} to Maxwell's equations, satisfying the estimate \eqref{eqn:CGO-R}. In the literature, such construction for systems \cite{CP, OS}, as well as for scalar elliptic equations \cite{SU} etc., heavily relies on the estimate of the Faddeev kernel. It is an integral operator $G_\zeta$, understood as the inverse of the second order operator $-(\Delta+2i\zeta\cdot\nabla)$ with $\zeta\in\C^\ndim\setminus\{0\}$, defined by
\[G_\zeta(f):=\mathcal{F}^{-1}\left(\frac{\mathcal{F}(f)(\xi)}{\xi^2+2\zeta\cdot\xi}\right)\]
where $\mathcal F$ denotes the Fourier transform. It was shown in the Proposition 3.6 of \cite{SU} that for $|\zeta|\geq c>0$ and $-1<\theta<0$, there exists a constant $C_{\theta,c}$ such that
\begin{equation}\label{eqn:FK-L2delta}
\|G_\zeta\|_{L^2_{\theta+1}\rightarrow L^2_\theta}\leq \frac{C_{\theta,c}}{|\zeta|}
\end{equation} 
where the weighted $L^2$-space is defined by
\[L^2_\theta:=\left\{f~\big|~\|f\|_\theta:= \|\langle x\rangle^\theta f\|_{L^2(\R^d)}<\infty\right\},\qquad \langle x\rangle:=(1+|x|^2)^{1/2}.\]
In general, the compactness estimate \eqref{eqn:FK-L2delta} is enough for a Neumann-series type of construction of $R_\zeta$ as in a CGO solution $E_\zeta=e^{ix\cdot\zeta}(\eta+R_\zeta)$ (e.g., see \cite{CP} for a direct construction of such), and to prove an estimate
\[\|R_\zeta\|_{H^\epsilon(\Omega;\C^\ndim)}\leq C|\zeta|^{\epsilon-1},\qquad \epsilon\in[0,2],\]
provided that $\eta$ is $O(1)$ for $|\zeta|\gg1$. %and the parameters (subtracting a constant) and their derivatives are compactly supported in $L^\infty(\R^d)$.
In another word, boundedness of $R_\zeta$ with respect to $|\zeta|$ can be obtained at most for its $H^1(\Omega;\C^\ndim)$-norm.  
To show point-wise boundedness as in Lemma \ref{lem:CGO}, especially when $\eta_\zeta$ is of $O(|\zeta|)$ as $|\zeta|\gg1$, we would need an estimate of $G_\zeta$ with higher regularity. Here we point out that similar results were proved in some weaker spaces like Besov spaces \cite{Br} and Bourgain type of spaces \cite{HT}.\\

Let $\Omega$ be a bounded domain in $\R^\ndim$ with a smooth boundary and $\Omega\subset\subset B_r$ for some ball of radius $r>0$. We assume that the background parameters $n_0,\sigma_0$ are as stated in Lemma \ref{lem:CGO}.

To our end, define the fractional operator 
\[(I-\Delta)^{l/2}:~f~\mapsto~\mathcal{F}^{-1}\left(\langle\xi\rangle^l\mathcal F(f)\right),\quad \langle\xi\rangle=(1+|\xi|^2)^{1/2}\]
and denote the weighted Sobolev space $H^l_\delta$ for $l\geq0$ as the completion of $C_0^\infty(\R^\ndim)$ with respect to the norm $\|\cdot\|_{H^l_\theta}$ defined as
\[\|f\|_{H^l_\theta}:=\|(I-\Delta)^{l/2}f\|_{L^2_\theta}.\] 
Notice that $-(\Delta+2i\zeta\cdot\nabla)$ and $(I-\Delta)^l$ commute, therefore \eqref{eqn:FK-L2delta} implies that for $|\zeta|\geq c>0$, 
\begin{equation}\label{eqn:FK-Hdelta}
%G_\zeta: H_{\delta+1}^l\,\rightarrow\, H_\delta^l,\quad 
\|G_\zeta\|_{H^l_{\theta+1}\rightarrow H^l_\theta}\leq \frac{C_{\theta,c}}{|\zeta|},\quad \textrm{ for all }\; l\geq0.
\end{equation}
Similar generalization applies to \cite[Lemma 2.11]{Na}, so we obtain
\begin{equation}\label{eqn:FK-H2}
 \|G_\zeta\|_{H^{l}_{\theta+1}\rightarrow H^{l+2}_\theta}\leq C|\zeta|.
\end{equation}
If $f\in H^{l}$ is compactly supported on $B_r$, \eqref{eqn:FK-H2} implies 
\begin{equation}\label{eqn:FK-H2-B}
\|G_\zeta f\|_{H^{l+2}(B_r)}\leq C|\zeta|\|f\|_{H^l(B_r)},
\end{equation}
and moreover, by interior regularity, the following estimate on a bounded domain is valid 
\begin{equation}\label{eqn:FK-Hs}
\|G_\zeta f\|_{H^{l}(B_r)}\leq C|\zeta|^{-1}\|f\|_{H^l(B_r)},\quad \|G_\zeta f\|_{H^{l+1}(B_r)}\leq C \|f\|_{H^l(B_r)}
\end{equation}
for all $l\geq 0$.
\\

The rest of this section contributes to the construction of $R_\zeta$ using \eqref{eqn:FK-Hdelta} and a similar argument as in \cite{CP}. To show that such argument is independent of the space dimension $\ndim$, % So is section 3!
throughout this section, we generalize our analysis for differential forms. \\

Let $L^2_\theta(\R^\ndim; \Lambda^m\R^\ndim)$ and $H^l_\theta(\R^\ndim; \Lambda^m\R^\ndim)$ denote spaces of $m$-forms whose components are in $L^2_\theta$ and $H^l_\theta$, respectively,
and let $H^l(B_r;\Lambda^m\R^d)$ denote the space of $m$-forms whose components are in $H^l(B)$. Corresponding norms are defined as the sum of the norms of component functions.

Let $l'>\ndim/2$. Recall that $k=\omega \sqrt{n_c}$ and $\gamma_0= q_0/k^2$ so that $\gamma_0-1\in H^{l'+2}(\R^\ndim; \Lambda^0\R^\ndim)\subset C^2(\R^\ndim;\Lambda^0\R^\ndim)$ is compactly supported on $B_r$. 
Let $ H^d (\R^\ndim; \Lambda^l \R^\ndim) $ %(resp., $H^\delta(G;\Lambda^l\R^\ndim)$) 
denote the space of $ u \in L^2 (\R^\ndim; \Lambda^l \R^\ndim) $ such that $ du \in L^2 (\R^\ndim; \Lambda^{l + 1} \R^\ndim), $ %(resp. $\delta u\in L^2(G; \Lambda^{l-1}\R^\ndim)$), 
endowed with the norm
\[ \norm{u}{}{H^d (\R^\ndim; \Lambda^l \R^\ndim)} = \left( \norm{u}{2}{L^2 (\R^\ndim; \Lambda^l \R^\ndim)} + \norm{du}{2}{L^2 (\R^\ndim; \Lambda^{l + 1} \R^\ndim)} \right)^{1/2}. \]
%\[\left(\mbox{resp., }\; \norm{u}{}{H^\delta (G; \Lambda^l \R^\ndim)} = \left( \norm{u}{2}{L^2 (G; \Lambda^l \R^\ndim)} + \norm{\delta u}{2}{L^2 (G; \Lambda^{l - 1} \R^\ndim)} \right)^{1/2}\right). \]
Then we can rewrite Maxwell's equations for the 1-form 
$E\in H^d(\R^\ndim;\Lambda^1\R^\ndim)$ %(See Appendix \ref{A} for corresponding definitions) 
as
\begin{equation}\label{eqn:ME-1-df}
\left\{\begin{split}&(\delta d-k^2\gamma_0)E=0, \\  &\delta(\gamma_0E)=0.\end{split}\right.
\end{equation}
where $d$ denotes the exterior derivative and $\delta$ denotes its adjoint. We define the conjugate operators on $C^1(\R^\ndim;\Lambda^l\R^\ndim)$
\[\wt d:=e^{-ix\cdot\zeta}\circ d\circ e^{ix\cdot\zeta}=d+i\zeta\wedge,\quad \wt \delta:=e^{-ix\cdot\zeta}\circ d\circ e^{ix\cdot\zeta}=\delta+(-1)^li\zeta\vee,\]
and write $E=e^{ix\cdot\zeta}(\wh\eta+\wh R)$, where $\wh \eta:=\gamma_0^{-\half}\eta$ and $\wh R:=\gamma_0^{-\half}R$.
Then \eqref{eqn:ME-1-df} implies
\begin{equation}\label{eqn:conj-1-df}\left\{\begin{split}
&\wt\delta\wt d \wh R=k^2\gamma_0\wh \eta+k^2\gamma_0\wh R-\wt\delta\wt d\wh\eta,\\ 
&\wt \delta(\gamma_0\wh R)=-\wt\delta(\gamma_0\wh\eta).
\end{split}\right.
\end{equation}
Denoting $\alpha:=d\gamma_0/\gamma_0$, the second equation above gives
\[\wt\delta\wh R=\alpha\vee (\wh\eta+\wh R)-\wt\delta\wh \eta.\]
Applying $\wt d$ and adding to the first equation of \eqref{eqn:conj-1-df}, we obtain
\begin{equation}\label{eqn:R0-df}
-\wt\Delta_H\wh R:=(\wt\delta\wt d+\wt d\wt\delta)\wh R=\wt d\big(\alpha\vee (\wh\eta+\wh R)\big)+k^2\gamma_0(\wh\eta+\wh R)-(\wt\delta\wt d+\wt d\wt\delta)\wh\eta,
\end{equation}
where $-\wt\Delta_H$ is the conjugate Hodge laplacian
\[-\wt\Delta_H=e^{-ix\cdot\zeta}\circ (-\Delta_H)\circ e^{ix\cdot\zeta}=-\Delta_H-2(i\zeta\vee d)+\langle \zeta,\zeta\rangle.\] %Here the Hodge Laplacian $-\Delta_H:=d\delta+\delta d$ and the operator $(i\zeta\vee d)$ are well defined in \ref{A:iden}. 
Suppose $\langle\zeta,\zeta\rangle=k^2$. 
By the identities for $1$-forms $u, v$ and $l$-form $w$  %\eqref{eqn:d-vee}, \eqref{eq:1v1wl} 
%\begin{equation}\label{eqn:d-vee}
\[d(u\vee v)=(u\vee d) v+(v\vee d) u+u\vee dv+v\vee du,\]
\[u \vee (v \wedge w) - v \wedge (u \vee w) = (-1)^l \inner{u}{v} w,\]
%\end{equation}
and $d\alpha=0$, we have 
\[
\wt d\big(\alpha\vee (\wh\eta+\wh R)\big)=\alpha\vee \wt d(\wh\eta+\wh R)+(\alpha\vee\wt d)(\wh\eta+\wh R)+\big((\wh\eta+\wh R)\vee d\big)\alpha.
\] 
Therefore \eqref{eqn:R0-df} becomes
\begin{equation}\label{eqn:R1-df}\begin{split}
-\big(\Delta_H+2(i\zeta\vee d)+(\alpha\vee\wt d)\big)\wh R=&\alpha\vee \wt d(\wh\eta+\wh R)+\big((\wh\eta+\wh R)\vee d\big)\alpha+k^2(\gamma_0-1)(\wh\eta+\wh R)\\
&+\big(\Delta_H+2(i\zeta\vee d)+(\alpha\vee\wt d)\big)\wh\eta.
\end{split}\end{equation}
For the operator on the left hand side, we have 
\[
\gamma_0^{\half}\circ-\big(\Delta_H-2(i\zeta\vee d)-(\alpha\vee\wt d)\big)\circ \gamma_0^{-\half}=
-\Delta_H-2(i\zeta\vee d)+\mathfrak{q},
\]
where $\mathfrak{q}:=\frac{1}{4}\langle\alpha,\alpha\rangle-\half\delta\alpha$.
Then, \eqref{eqn:R1-df} gives
\begin{equation}\label{eqn:R2-df}\begin{split}
-\big(\Delta_H+2(i\zeta\vee d)\big)R=&\gamma_0^\half\alpha\vee \wt d(\wh\eta+\wh R)+(R\vee d)\alpha+k^2(\gamma_0-1)R-\mathfrak qR\\
&+(\eta\vee d)\alpha+k^2(\gamma_0-1)\eta-\mathfrak q\eta.
\end{split}\end{equation}
Applying $G_\zeta$ yields the integral equation 
\begin{equation}\label{eqn:R2-df-int}\begin{split}
R=&G_\zeta\big[\gamma_0^\half\alpha\vee\wt d(\wh\eta+\wh R)\big]+G_\zeta\left[(R\vee d)\alpha+k^2(\gamma_0-1)R-\mathfrak qR\right]\\
&+G_\zeta\left[(\eta\vee d)\alpha+k^2(\gamma_0-1)\eta-\mathfrak q\eta\right]
\end{split}\end{equation}
\begin{lemma}\label{lem:4-1}
Suppose $R\in H^{l'}_\delta(\R^\ndim;\Lambda^1\R^\ndim)$ is a solution to the integral equation \eqref{eqn:R2-df-int}. Then $\wh R:=\gamma_0^{-\half}R\in C^2(\R^\ndim;\Lambda^1\R^\ndim)$ satisfies \eqref{eqn:conj-1-df}, that is, $E=\gamma_0^{-\half}e^{ix\cdot\zeta}(\eta+R)$ satisfies \eqref{eqn:ME-1-df}.
\end{lemma}
\begin{proof}
First we remark that $c^{-1}<\gamma_0<c$ for some constant $c>1$, then it can be shown that $\alpha=d\gamma_0/\gamma_0$ and $\gamma_0^{1/2}\alpha$ belong to $H^{l'+1}(\R^\ndim;\Lambda^1\R^\ndim)$ and they are compactly supported on $B_r$. %(since $1/\gamma_0$ is bounded from above and below.) 
Then since $H^{l}$ for $l\geq \ndim/2$ is an algebra, it is easy to see that $\mathfrak q=\frac{1}{4}\langle\alpha,\alpha\rangle-\frac{1}{2}\delta\alpha \in H^{l'}(\R^\ndim;\Lambda^0\R^\ndim)$ is also compactly supported on $B_r$. 
Then one can easily see that the right hand side of \eqref{eqn:R2-df} is in $H^{l'}_{\theta+1}$, by \eqref{eqn:FK-H2} and Sobolev embedding, we have $R\in C^2(\R^\ndim;\Lambda^1\R^\ndim)$. 

Hence, the above derivation from \eqref{eqn:R0-df} to \eqref{eqn:R2-df-int} can be reversed to obtain that $\wh R\in C^2(\R^\ndim;\Lambda^1\R^\ndim)$ satisfies \eqref{eqn:R0-df}, that is,
\begin{equation}\label{eqn:R0-df-1}
\wt\delta\wt d(\wh\eta+\wh R)=\wt d\left(-\wt\delta(\wh \eta+\wh R)+\alpha\vee(\wh\eta+\wh R)\right) + k^2\gamma_0(\wh\eta+\wh R).
\end{equation}
Denote the 0-form
\[u:=-\wt\delta(\wh\eta+\wh R)+\alpha\vee(\wh\eta+\wh R)=-\gamma_0^{-1}\wt\delta\left(\gamma_0(\wh\eta+\wh R)\right).\]
Applying $\wt\delta$ to \eqref{eqn:R0-df-1} and adding to $\wt\delta u=0$ (for it is a 0-form), we have eventually
\[-(\Delta_H+2(i\zeta\vee d))u=k^2(\gamma_0-1)u.\]
This equation by \eqref{eqn:FK-Hdelta} admits only the trivial solution. Hence $u=0$, which implies the second equation of \eqref{eqn:conj-1-df}, furthermore \eqref{eqn:R0-df-1} becomes the first equation of \eqref{eqn:conj-1-df}. This completes the proof. 
\end{proof}

Next, to solve the integral equation \eqref{eqn:R2-df-int}, we define $Q:=\wt d(\wh\eta+\wh R)$. First note that $\wt dQ=0$. % since $\wt d\wt d=\wt{dd}=0$.
Then, by the first equation of \eqref{eqn:conj-1-df},
\[\wt \delta Q=k^2\gamma_0(\wh\eta+\wh R),\]
which gives 
\[\wt d\wt\delta Q=k^2\gamma_0^\half \alpha\wedge(\eta+R)+k^2\gamma_0Q.
\]
Together with \eqref{eqn:R2-df}, we obtain a system of $R$ and $Q$
\begin{equation}\label{eqn:RQ}\left\{\begin{split}
-\big(\Delta_H+2(i\zeta\vee d)\big)R=&\gamma_0^\half\alpha\vee Q+(R\vee d)\alpha+k^2(\gamma_0-1)R-\mathfrak qR\\
&+(\eta\vee d)\alpha+k^2(\gamma_0-1)\eta-\mathfrak q\eta\\
-\big(\Delta_H+2(i\zeta\vee d)\big)Q=&k^2\gamma_0^\half\alpha\wedge R+k^2(\gamma_0-1)Q+k^2\gamma_0^\half\alpha\wedge\eta.
\end{split}\right.\end{equation}

\begin{lemma}\label{lem:RQ}
%Suppose $\gamma_0-1\in H^{l'+2}(\R^d;\Lambda^0\R^d)$ ($l'>d/2$) is compactly supported on a ball $B$ of radius $R>0$ s.t. $\overline{\Omega}\subset B$ and $\gamma_0\neq 0$ everywhere. 
For $|\zeta|$ large enough, we have that \eqref{eqn:RQ} admits a unique solution $R\in H^{l'}_\theta(\R^\ndim;\Lambda^1\R^\ndim)$, $Q\in H^{l'}_\theta(\R^\ndim;\Lambda^2\R^\ndim)$ satisfying
\begin{equation}\label{eqn:est-RQ}
\|R\|_{H^{l'}_\theta(\R^\ndim;\Lambda^1\R^\ndim)}+\|Q\|_{H^{l'}_\theta(\R^\ndim;\Lambda^2\R^\ndim)}\leq C|\zeta|^{-1}|\eta|.
\end{equation}
\end{lemma}
\begin{proof}

The solution of \eqref{eqn:RQ} is constructed using the decomposition
\[R=\sum_{m=0}^\infty R_m,\quad Q=\sum_{m=0}^\infty Q_m\]
where the first pair $(R_0, Q_0)$ satisfies
\begin{equation}\label{eqn:RQ-0}\left\{\begin{split}
-\big(\Delta_H+2(i\zeta\vee d)\big)R_0=&(\eta\vee d)\alpha+k^2(\gamma_0-1)\eta-\mathfrak q\eta,\\
-\big(\Delta_H+2(i\zeta\vee d)\big)Q_0=&k^2\gamma_0^\half\alpha\wedge\eta,
\end{split}\right.\end{equation}
and for $m>0$,
\begin{equation}\label{eqn:RQ-m}\left\{\begin{split}
-\big(\Delta_H+2(i\zeta\vee d)\big)R_m=&M_1(R_{m-1},Q_{m-1})\\
:=&\gamma_0^\half\alpha\vee Q_{m-1}+(R_{m-1}\vee d)\alpha+k^2(\gamma_0-1)R_{m-1}-\mathfrak q R_{m-1},\\
-\big(\Delta_H+2(i\zeta\vee d)\big)Q_m
=&M_2(R_{m-1},Q_{m-1})\\
:=&k^2\gamma^\half\alpha\wedge R_{m-1}+k^2(\gamma_0-1)Q_{m-1}.
\end{split}\right.\end{equation}

By the beginning remark, the right hand side forms of \eqref{eqn:RQ-0} are both in $H^{l'}$ and compactly supported on $B_r$, hence in $H^{l'}_{\delta+1}$. Then \eqref{eqn:FK-Hdelta} implies that for $|\zeta|$ large enough, $R_0\in H^{l'}_\theta(\R^\ndim;\Lambda^1\R^\ndim)$ and $Q_0\in H^{l'}_\theta(\R^\ndim;\Lambda^2\R^\ndim)$. Moreover, one has
\[\|R_0\|_{H^{l'}_\theta}+\|Q_0\|_{H^{l'}_\theta}\leq C(\theta,\gamma_0)|\zeta|^{-1}|\eta|\]
for some positive constant $C(\theta,\gamma_0)$.

Suppose that $R_{m-1}\in H^{l'}_\theta(\R^\ndim;\Lambda^1\R^\ndim)$ and $Q_{m-1}\in H^{l'}_\theta(\R^\ndim;\Lambda^2\R^\ndim)$. Notice that the operators $M_1$ and $M_2$ defined in \eqref{eqn:RQ-m} are essentially multiplications of component functions of forms by compactly supported $H^{l'}$ parameters. By a similar argument in proving \cite[equation (23)]{BU-IP-10}, we can show that $M_1(R_{m-1},Q_{m-1})\in H^{l'}_{\theta+1}(\R^\ndim;\Lambda^1\R^\ndim)$, $M_2(R_{m-1},Q_{m-1})\in H^{l'}_{\theta+1}(\R^\ndim;\Lambda^2\R^\ndim)$ and 
\[\|M_1(R_{m-1},Q_{m-1})\|_{H^{l'}_{\theta+1}}+\|M_2(R_{m-1},Q_{m-1})\|_{H^{l'}_{\theta+1}}\leq C_{\gamma_0}\left(\|R_{m-1}\|_{H^{l'}_\theta}+\|Q_{m-1}\|_{H^{l'}_\theta}\right).\]
Then by \eqref{eqn:FK-Hdelta}, we obtain
\[\|R_m\|_{H^{l'}_\theta}+\|Q_m\|_{H^{l'}_\theta}\leq C|\zeta|^{-1}\left(\|R_{m-1}\|_{H^{l'}_\theta}+\|Q_{m-1}\|_{H^{l'}_\theta}\right)\]
By taking $|\zeta|$ large such that $C|\zeta|^{-1}<\half$, above estimates yield
\[\|R_m\|_{H^{l'}_\theta}+\|Q_m\|_{H^{l'}_\theta}\leq \frac{1}{2^m}C|\zeta|^{-1}|\eta|,\quad m\geq0.\]
Summing up the geometric series proves \eqref{eqn:est-RQ}. 

The uniqueness of the solution is a consequence of the unique solvability of $-(\Delta_H+2(i\zeta\vee d))$ and estimate \eqref{eqn:FK-Hdelta}. 
\end{proof}

%\begin{remark}
%The proof of Lemma \ref{lem:RQ} shows that whenever we know $R$, $Q$ is a solution, the regularity is free. 
%\end{remark}

It remains to show that $R$ also uniquely solves \eqref{eqn:R2-df-int} in $H^{l'}_\theta$. 

\begin{lemma}\label{lem:4-3}
For $|\zeta|$ large enough, there exists a unique solution $R\in H^{l'}_\theta(\R^\ndim;\Lambda^1\R^\ndim)$ to \eqref{eqn:R2-df-int}. Moreover, it satisfies the estimate
\begin{equation}\label{eqn:est-R}
\|R\|_{H^{l'}_\theta}\leq C|\zeta|^{-1}|\eta|.
\end{equation}
\end{lemma}
\begin{proof} 
The proof is similar to that of \cite[Theorem 3.1]{CP}. By Fredholm alternative, we need to show \eqref{eqn:R2-df-int}, rewritten as
\[(I-K)R=G_\zeta\big[\gamma_0^\half\alpha\vee\wt d\wh\eta+(\eta\vee d)\alpha+k^2(\gamma_0-1)\eta-\mathfrak q\eta\big]:=f\]
is of Fredholm type in $L^2(B_r;\Lambda^1\R^\ndim)$ and the kernel is zero, where 
\[K(R):=G_\zeta\big[\gamma_0^\half\alpha\vee \wt d\wh R\big]-G_\zeta\left[(R\vee d)\alpha+k^2(\gamma_0-1)R-\mathfrak qR\right].\]

By \eqref{eqn:FK-H2-B} and \eqref{eqn:FK-Hs}, both terms in $K$ are bounded  (the constant is linear in $|\zeta|$) linear operators from $L^2(B_r; \Lambda^1\R^\ndim)$ to $H^1(B_r; \Lambda^1\R^\ndim)$. Since $H^1(B_r; \Lambda^1\R^\ndim)$ is compactly embedded in $L^2(B_r; \Lambda^1\R^\ndim)$, $K$ is compact. 
%Notice here, we don't need the boundedness with respect to $|\zeta|$ to show the compactness. Therefore, using \eqref{eqn:FK-H2-B} is fine. 
This proves that $I-K$ is Fredholm type. 

To show that the kernel of $I-K$ in $L^2(B_r;\Lambda^1\R^\ndim)$ is zero, consider the solution $R^h\in L^2(B_r; \Lambda^1\R^\ndim)$ to the homogeneous equation $R^h=K(R^h)$. Since all the parameters are supported on $B_r$, $R^h$ can be extended to a solution in $L^2_\theta(\R^\ndim;\Lambda^1\R^\ndim)$. Then the extension satisfies the homogeneous integral equations corresponding to \eqref{eqn:RQ}. By the previous Lemma \ref{lem:RQ} (in which the uniqueness can also be shown in $L^2_\theta$), we have $R^h=0$. 

Therefore, $I-K$ is uniquely solvable in $L^2(B_r;\Lambda^1\R^\ndim)$. It is not hard to see that $f\in L^2(B_r;\Lambda^1\R^\ndim)$ with bounded norm in $|\zeta|$. This implies that
\[R=(I-K)^{-1}f\in L^2(B_r;\Lambda^1\R^\ndim).\]
Extend  and plug it into the right hand side of \eqref{eqn:R2-df-int}. Note that the parameters are all compactly supported in $B_r$ and $f\in H^{l'+2}_{\theta}(\R^\ndim;\Lambda^1\R^\ndim)$, we have $R\in H^1_\theta(\R^\ndim;\Lambda^1\R^\ndim)$ solves \eqref{eqn:R2-df-int} on $\R^\ndim$. %(atcually, $R\in H^2_\theta$)

Finally,  this $H^1_\theta$ solution $R$ to \eqref{eqn:R2-df-int} and $Q:=\wt d(\wh\eta+\wh R)\in L^2_\theta(\R^\ndim;\Lambda^2\R^\ndim)$ must satisfy \eqref{eqn:RQ}, therefore by Lemma \ref{lem:RQ}, we obtain $R\in H^{l'}_\theta(\R^\ndim;\Lambda^1\R^\ndim)$, the uniqueness and the estimate \eqref{eqn:est-R}.
\end{proof}

\begin{proof}[Proof of Lemma \ref{lem:CGO}] Let $\zeta$ be as in \eqref{eqn:zeta} and $\eta=\eta_\zeta$ be as in \eqref{eqn:eta-zeta}. By Lemma \ref{lem:4-1} and Lemma \ref{lem:4-3}, we obtain, for $|\zeta|$ large enough, a solution $E_\zeta\in H^1_{\textrm{loc}}(\R^\ndim;\Lambda^1\R^\ndim)$ to Maxwell's equations \eqref{eqn:ME-1-df}, of the form \eqref{eqn:E-CGO} with $R_\zeta\in H^{l'}_\theta(\R^\ndim;\Lambda^1\R^\ndim)$ satisfying 
\[\|R_\zeta\|_{H^{l'}_\theta}\leq C|\zeta|^{-1}|\eta_\zeta|\leq C.\]
This implies
\[\|R_\zeta\|_{H^{l'}(\Omega;\Lambda^1\R^\ndim)}\leq C.\]
Then by Sobolev embedding, we have $R_\zeta$ bounded point-wise by $C$.
\end{proof}

\end{document}